\documentclass[twoside,11pt]{article}
\usepackage[utf8]{inputenc}
\usepackage[english]{babel} 
\usepackage{amsmath,amssymb,amsthm,amsfonts}
\usepackage{stmaryrd} 
\usepackage{mathabx} 
\usepackage{bbm}
\usepackage{graphicx}
\usepackage[left=2.5cm,right=2.5cm,top=2cm,bottom=2cm]{geometry}
\usepackage{hyperref}
\usepackage[d]{esvect}

\usepackage{enumitem}

\theoremstyle{plain}
\newtheorem{theorem}{Theorem}
\newtheorem{lemma}[theorem]{Lemma}
\newtheorem{corollary}[theorem]{Corollary}

\newtheorem{question}[theorem]{Question}

\newtheorem*{theorem*}{Theorem}

\theoremstyle{definition}

\theoremstyle{remark}
\newtheorem{remark}[theorem]{Remark}
\newtheorem*{remark*}{Remark}

\def\N{\ensuremath{\mathbb{N}}}

\def\R{\ensuremath{\mathbb{R}}}
\def\Z{\ensuremath{\mathbb{Z}}}

\def\ep{\varepsilon}

\def\E{\ensuremath{\mathbf{E}}}
\def\P{\ensuremath{\mathbf{P}}}

\def\Ind{\ensuremath{\mathbbm{1}}}

\renewcommand\Re{\operatorname{Re}}

\def\to{\rightarrow}

\def\tand{\ensuremath{\text{ and }}}
\def\tif{\ensuremath{\text{ if }}}
\def\tas{\ensuremath{\text{ as }}}

\title{The $\lambda$-invariant measures of subcritical Bienaym\'e--Galton--Watson processes}
\author{Pascal Maillard\thanks{
Laboratoire de Math\'ematiques d'Orsay, Univ. Paris-Sud, CNRS, Universit\'e Paris-Saclay, 91405 Orsay Cedex, France. e-mail: pascal dot maillard at u-psud dot fr. Partially supported by Grant ANR-14-CE25-0014 (ANR GRAAL).}}

\begin{document}
\maketitle
\begin{abstract} 
A $\lambda$-invariant measure of a sub-Markov chain is a left eigenvector of its transition matrix of eigenvalue $\lambda$. In this article, we give an explicit integral representation of the $\lambda$-invariant measures of subcritical Bienaym\'e--Galton--Watson processes killed upon extinction, i.e.\ upon hitting the origin. In particular, this characterizes all quasi-stationary distributions of these processes. Our formula extends the Kesten--Spitzer formula for the (1-)invariant measures of such a process and can be interpreted as the identification of its minimal $\lambda$-Martin entrance boundary for all $\lambda$. In the particular case of quasi-stationary distributions, we also present an equivalent characterization in terms of semi-stable subordinators.

Unlike Kesten and Spitzer's arguments, our proofs are elementary and do not rely on Martin boundary theory.

\bigskip


\bigskip

\end{abstract}

\section{Results}

Let $Z = (Z_n)_{n\ge0}$ be a subcritical Bienaym\'e--Galton--Watson (BGW) process with offspring distribution of mean $m\in(0,1)$. Denote by $P$ the restriction of its transition matrix to $\N^*=\{1,2,\ldots\}$. Then $P$ is a sub-stochastic matrix, the transition matrix of the sub-Markov process \{$Z$ killed upon hitting $0$\}. A measure\footnote{Throughout the article, all measures are assumed to be locally finite unless explicitly stated.} $\nu$ on $\N^*$ is called a $\lambda$-invariant measure for $Z$ if it is a left eigenvector\footnote{As usual, we consider measures as row vectors and functions as column vectors.} of $P$ of eigenvalue $\lambda$, i.e. if 
\begin{equation}
\label{eq:nu_P}
\nu P = \lambda \nu.
\end{equation}
In terms of generating functions, if $F(z)$ denotes the generating function of the offspring distribution and $G(z) = \sum_{k=1}^\infty \nu(k)z^k$ the generating function of the measure $\nu$, then, supposing that $G(z)$ is finite for all $|z|<1$ (a fact which follows from Lemma~\ref{lem:tail} below), \eqref{eq:nu_P} is equivalent to
\begin{equation}
 \label{eq:gf_measure}
 G(F(z)) - G(F(0)) = \lambda G(z),\quad|z|<1.
\end{equation}

For $x\in\N^*$ denote by $\P_x$ the law of the process $Z$ starting from $Z_0 = x$ and by $\E_x$ expectation with respect to $\P_x$. Furthermore, for a measure $\nu$ on $\N^*$, write $\|\nu\| = \nu(\N^*)$. The following limit, called the \emph{Yaglom limit}, is known to exist \cite{Heathcote1967,Joffe1967} (see also \cite[p.\ 16]{Athreya1972}):
\begin{equation}
\label{eq:yaglom}
\nu_{\text{min}} = \lim_{n\to\infty} \P_1(Z_n \in \cdot\,|\,Z_n > 0) = \lim_{n\to\infty} \frac{\delta_1 P^n}{\|\delta_1 P^n\|}, 
\end{equation}
where the limit holds in the weak topology of measures on $\N^*$. Furthermore, the probability measure $\nu_{\text{min}}$ satisfies \eqref{eq:nu_P} with $\lambda = m$, i.e. it is an $m$-invariant probability measure of the process. In particular (see also \cite[Proposition~5]{Meleard2012}),
\begin{equation}
\label{eq:pn}
 \frac{\P_1(Z_{n+1} > 0)}{\P_1(Z_n>0)} = \frac{\|\delta_1 P^{n+1}\|}{\|\delta_1 P^n\|} \to m,\quad\tas n\to\infty.
\end{equation}

We denote the generating function of the probability measure $\nu_{\text{min}}$ by 
\[
H(z)= \sum_{n=1}^\infty \nu_{\text{min}}(n)z^n,\quad |z|\le 1.
\]
%
Our main theorem is the following, which identifies all $\lambda$-invariant measures of the BGW process $(Z_n)_{n\ge 0}$.
\begin{theorem}
\begin{enumerate}
 \item There exist no non-trivial (i.e. $\not\equiv 0$) $\lambda$-invariant measures for $Z$ with $\lambda < m$.
 \item The only $m$-invariant measures of $Z$ are multiples of the Yaglom limit $\nu_{\text{min}}$.
 \item Let $\alpha\in(-\infty,1)$. A measure $\nu$ on $\N^*$ is an $m^\alpha$-invariant measure for $Z$ if and only if its generating function $G(z) = \sum \nu(n)z^n$ satisfies
\begin{equation}
\label{eq:G_rep}
G(z) = \int_0^\infty (e^{(H(z)-1)x}-e^{-x})\frac 1 {x^\alpha}\,\Lambda(dx),\quad |z|<1,
\end{equation}
where $\Lambda$ is a locally finite measure on $(0,\infty)$ satisfying $\Lambda(A) = \Lambda(mA)$ for every Borel set $A\subset(0,\infty)$. The measure $\Lambda$ is uniquely determined from $\nu$. Moreover, for every such measure $\Lambda$, \eqref{eq:G_rep} defines the generating function of an $m^\alpha$-invariant measure for $Z$ with radius of convergence at least 1.
\end{enumerate}
\label{th:main}
\end{theorem}

\begin{remark}
 In the proof of Theorem~\ref{th:main}, the measure $x^{-\alpha}\Lambda(dx)$ will be constructed as the vague limit ($n\to\infty$) of the measures $\mu_n$ on $(0,\infty)$ defined by $\mu_n(A) = m^{-\alpha n}\nu(p_n^{-1}A)$, where $p_n = \P_1(Z_n > 0)$ and $A\subset(0,\infty)$ Borel.
\end{remark}

%
\begin{remark}
We will give an overview over the existing literature in Section~\ref{sec:history} but mention already here that Formula~\eqref{eq:G_rep} was obtained, in a slightly different form, for $\alpha=0$ and $F(z) = 1-m(1-z)$ (the ``pure death case'') by Kesten and Spitzer \cite{Spitzer1967} (giving credit to H.~Dinges for deriving it independently). It was later shown by Hoppe~\cite{Hoppe1977} that the case of general and even multitype offspring distributions (but still $\alpha=0$) can be reduced to the pure death case. One could adapt Hoppe's arguments for $\alpha\ne 0$, but we do not show this here. 
\end{remark}

\paragraph{Quasi-stationary distributions.}
A $\lambda$-invariant \emph{probability} measure $\nu$ (i.e., $\|\nu\| = 1$) is also called a \emph{quasi-stationary distribution} (QSD) of the process $Z$ with eigenvalue $\lambda$. The following result easily follows from Theorem~\ref{th:main}:

\begin{theorem}
 \label{th:qsd}
 A $\lambda$-invariant measure $\nu$ of the process $Z$ is finite if and only if $\lambda < 1$. In particular, $\nu$ is a QSD with eigenvalue $\lambda$ of the process $Z$ if and only if either
 \begin{enumerate}
  \item $\lambda = m$ and $\nu = \nu_{\text{min}}$, or
  \item $\lambda = m^\alpha$ for some $\alpha\in(0,1)$ and the generating function $G(z) = \sum \nu(n)z^n$ satisfies \eqref{eq:G_rep}, where $\Lambda$ is a locally finite measure on $(0,\infty)$ satisfying 
\begin{enumerate}
 \item $\Lambda(A) = \Lambda(mA)$ for every Borel set $A\subset(0,\infty)$ and
 \item $\int_0^\infty (1-e^{-x})x^{-\alpha}\,\Lambda(dx) = 1$.
\end{enumerate}
Furthermore, the measure $\Lambda$ in \eqref{eq:G_rep} is uniquely determined from $\nu$.
 \end{enumerate}
\end{theorem}

%
%
%

%
%

\begin{remark}
If $\Lambda(dx) = \frac 1 x\,dx$ in the above theorem, then $G(z) = 1-(1-H(z))^\alpha$, see Remark~\ref{rem:formulae} below. These QSD were found by Seneta and Vere-Jones in their seminal paper on QSD of Markov chains on countably infinite state spaces \cite{Seneta1966}. Rubin and Vere-Jones \cite{Rubin1968} showed later that these QSD are the only ones with regularly varying tails and furthermore, every distribution $\nu$ on $\N^*$ with a tail of the form $\nu([x,\infty)) = x^{-\alpha}L(x)$ for a slowly varying function $L(x)$ is in the domain of attraction of the above QSD (see also \cite{Seneta1971} for an analogous result for $1$-invariant measures).
\end{remark}

\begin{remark}
 The fact that the Yaglom limit $\nu_{\text{min}}$ as defined in \eqref{eq:yaglom} is the QSD of smallest eigenvalue is a general fact \cite[p515]{Ferrari1995a}. Furthermore, the fact that it is the unique QSD with eigenvalue $m$ is classic in our case \cite{Heathcote1967,Joffe1967}.
\end{remark}
%
%

\begin{remark}
QSD of BGW processes (and Formula~\eqref{eq:G_rep}) appear in a recent article by H\'enard and the author on random trees invariant under Bernoulli edge contraction \cite{Henard2014}.
\end{remark}

\paragraph{Continuous-time BGW processes.}

$\lambda$-invariant measures can be defined analogous\-ly for a subcritical \emph{continuous-time} BGW process $(Z_t)_{t\ge0}$. Let $L$  and $(P_t)_{t\ge0}$ be its associated infinitesimal generator and semigroup, respectively, restricted to $\N^*$. We say that a measure $\nu$ on $\N^*$ is a $\lambda$-invariant measure of the process $(Z_t)_{t\ge0}$ if $\nu L = -\lambda \nu$, or equivalently, if $\nu P_t = e^{-\lambda t} \nu$ for every $t\ge0$. In this case, $\nu$ is also a $e^{-\lambda r}$-invariant measure of the embedded chain $(Z_{rn})_{n\ge0}$, for every $r>0$. The measure $\Lambda$ from Theorem~\ref{th:main} then satisfies $\Lambda(A) = \Lambda(rA)$ for every Borel set $A$ and every $r>0$, hence $\Lambda(dx) = \frac 1 x\,dx$. We therefore have the following corollary to Theorem~\ref{th:main}, which also follows (in the pure death case) from results for general birth-and-death chains \cite{Cavender1978}.

\begin{corollary}
\label{cor:continuous}
 Let $(Z_t)_{t\ge0}$ be a subcritical continuous-time BGW process and let $m>0$ such that for all $t\ge0$, $\E_1[Z_t] = e^{-mt}$. 
\begin{enumerate}
\item There exist no non-trivial (i.e. $\not\equiv 0$) $\lambda$-invariant measures for $(Z_t)_{t\ge0}$ with $\lambda > m$.
 \item The only $m$-invariant measures of $(Z_t)_{t\ge0}$ are multiples of the Yaglom limit $\nu_{\text{min}}$.
\item	For every $\lambda < m$, the $\lambda$-invariant measures of $(Z_t)_{t\ge0}$ are exactly the multiples of the measure whose generating function is given by \eqref{eq:G_rep} with $\alpha = \lambda/m$ and $\Lambda(dx) = \frac 1 x\,dx$ if $\lambda < m$. More explicitly, $G$ is the generating function of a $\lambda$-invariant measure if and only if there exists $c\ge0$, such that
\[
G(z) = c\times
\begin{cases}
1- (1-H(z))^\alpha, & \alpha > 0\\
-\log (1-H(z)), & \alpha = 0\\
(1-H(z))^\alpha - 1, & \alpha < 0.
\end{cases}
\]
\end{enumerate}
In particular, the only QSD of the process $(Z_t)_{t\ge0}$ with eigenvalue $\alpha m$, $\alpha\in(0,1]$ is the probability measure with generating function $1-(1-H(z))^\alpha$.
\end{corollary}
\begin{remark}
A similar phenomenon happens for \emph{continuous-state} branching processes, see \cite{Lambert2007}.
\end{remark}
\begin{remark}
\label{rem:formulae}
The explicit formulae in Corollary~\ref{cor:continuous} are obtained from the following well-known equality which we recall for convenience:
\begin{equation}
\forall a\in(0,1)\,\forall \alpha< 1: \int_0^\infty (e^{-ax} - e^{-x})\frac{dx}{x^{\alpha+1}} =
\begin{cases}
\Gamma(-\alpha)(a^\alpha - 1), & \alpha \ne 0\\
-\log a, & \alpha = 0.
\end{cases}
\label{eq:gamma}
\end{equation}
An easy proof goes by noting that for each $a\in(0,1)$, both sides of the equation define analytic functions on the half-plane $\{\Re \alpha < 1\}$ and agree on $\{\Re \alpha < 0\}$ as can easily be checked by calculating the two Euler integrals. The case $\alpha=0$ is also a special case of Frullani's integral.
\end{remark}

\paragraph{$\lambda$-invariant measures of the process which is not killed at the origin.}

Say that a measure $\nu$ on $\N = \{0,1,\ldots\}$ is a \emph{true} $\lambda$-invariant measure for $Z$ if it is a $\lambda$-invariant measure for the \emph{non-killed} process. In other words, if $P_0$ denotes the transition matrix of the BGW process $Z$ on $\N$, a measure $\nu$ on $\N$ is a true $\lambda$-invariant measure for $Z$ if and only if $\nu P_0 = \lambda \nu$, or equivalently, if its generating function $G$ satisfies $G(F(z)) = \lambda G(z)$ for every $|z|<1$.
Of course, since $0$ is an absorbing state for the process, there are no true $\lambda$-invariant measures for $\lambda < 1$, and for $\lambda = 1$ the only true (1-)invariant measures are the multiples of $\delta_0$ (see e.g. \cite[p.\,67]{Athreya1972}). However, for $\lambda > 1$ the $\lambda$-invariant measures from Theorem~\ref{th:main} all extend to true $\lambda$-invariant measures. In fact, we have the following analogue of Theorem~\ref{th:main}:
\begin{theorem}
\begin{enumerate}
	\item There exist no non-trivial (i.e. $\not\equiv 0$) true $\lambda$-invariant measures for $Z$ with $\lambda < 1$.
	\item The only true (1-)invariant measures for $Z$ are multiples of $\delta_0$.
  \item Let $\alpha< 0$. A measure $\nu$ on $\N$ is a true $m^\alpha$-invariant measure for $Z$ if and only if its generating function $G(z) = \sum \nu(n)z^n$ satisfies
		\begin{equation}
		\label{eq:G_rep2}
		G(z) = \int_0^\infty e^{(H(z)-1)x}\frac 1 {x^\alpha}\,\Lambda(dx),\quad |z|<1,
		\end{equation}
		where $\Lambda$ is a locally finite measure on $(0,\infty)$ satisfying $\Lambda(A) = \Lambda(mA)$ for every Borel set $A\subset(0,\infty)$. The measure $\Lambda$ is uniquely determined from $\nu$. Moreover, for every such measure $\Lambda$, \eqref{eq:G_rep} defines the generating function of a true $m^\alpha$-invariant measure for $Z$ with radius of convergence at least 1.
\end{enumerate}
\label{th:true_main}
\end{theorem}

For a subcritical \emph{continuous-time} BGW process, we can define true $\lambda$-invariant measures analoguously as above. The analogue of Corollary~\ref{cor:continuous} is then the following:
\begin{corollary}
Let $(Z_t)_{t\ge0}$ be a subcritical continuous-time BGW process and let $m>0$ such that for all $t\ge0$, $\E_1[Z_t] = e^{-mt}$. 
\begin{enumerate}
	\item There exist no non-trivial (i.e. $\not\equiv 0$) $\lambda$-invariant measures for $(Z_t)_{t\ge0}$ with $\lambda > 0$.
	\item The only true (0-)invariant measures for $(Z_t)_{t\ge0}$ are the multiples of $\delta_0$.
	\item For $\lambda < 0$, the true $\lambda$-invariant measures for $(Z_t)_{t\ge0}$ are exactly the multiples of the one given by \eqref{eq:G_rep2} with $\alpha = \lambda/m$ and $\Lambda(dx) = \frac 1 x\,dx$, i.e. the measures with generating functions $G(z) = c(1-H(z))^\alpha$, $c\ge 0$.
	\end{enumerate}
\end{corollary}

\paragraph{Overview of the article.}
The remainder of the article is organized as follows: Theorems~\ref{th:main}, \ref{th:qsd} and \ref{th:true_main} are proven in Section~\ref{sec:proof}. Section~\ref{sec:discussion} is an extended discussion consisting of the following three parts. In Section~\ref{sec:martin}, we interpret Theorem~\ref{th:main} in the light of Martin boundary theory. Section~\ref{sec:subordinators} gives a probabilistic interpretation of the QSD from Theorem~\ref{th:qsd} in terms of semi-stable subordinators. In Section~\ref{sec:history}, we review the existing literature on $\lambda$-invariant measures of BGW processes. 

\paragraph{Notation.} Throughout the article, a statement involving an undefined variable $z$ is meant to hold (at least) for every $z\in(0,1)$.

\paragraph{Acknowledgments.}
I thank Olivier H\'enard for an extremely fruitful collaboration \cite{Henard2014} from which this article arose.  I also thank Alano Ancona and Vadim Kaimanovich for useful discussions about Martin boundary theory. An anonymous referee has made several valuable suggestions improving the presentation of the article.

\section{Proofs}
\label{sec:proof}

\def\Poi{\operatorname{Poi}}

We start with three simple lemmas.
\begin{lemma}
 \label{lem:ineq}
 Let $z,w\in(0,1)$, $z\ne w$. Then for $p>0$ small enough and for all $x\ge 0$,
 \[
  (1-p(1-z))^{x/p} - (1-p)^{x/p} \quad \begin{matrix} \ge \\ \le \end{matrix} \quad e^{(w-1)x}-e^{-x} \quad \begin{matrix} \tif w<z\\ \tif w>z\end{matrix}
 \]
\end{lemma}
\begin{proof}
Let $z,w\in(0,1)$, $z\ne w$.
If $w<z$, then $1-p(1-z) = 1+p(z-1) \ge e^{(w-1)p}$ for all small enough $p>0$. Furthermore, $1-p \le e^{-p}$ for all $p>0$. This implies the first inequality.

If $w>z$, then $1-p \ge e^{(z-w-1)p}$ for all small enough $p>0$. Furthermore, $1-p(1-z) \le e^{(z-1)p}$ for all $p>0$. Hence, for $p>0$ small enough and $x\ge 0$.
\[
 (1-p(1-z))^{x/p} - (1-p)^{x/p} \le e^{(z-1)x} - e^{(z-w-1)x} = e^{(z-w)x}(e^{(w-1)x}-e^{-x}) \le e^{(w-1)x}-e^{-x}.
\]
This shows the second inequality and thus finishes the proof of the lemma.
\end{proof}

\begin{lemma}
 \label{lem:tail}
 Let $\nu$ be an $m^\alpha$-invariant measure of $Z$, with $\alpha\in\R$. Set $M = m^{-1}$. Then for every $\beta < \alpha$, there exists $C<\infty$, such that 
 \[
  \nu([M^n,M^{n+1})) \le C m^{\beta n},\quad\forall n\in\N.
 \]
As a consequence, we have for every $\beta < \alpha$, for some $C<\infty$, for every $x\ge 1$,
\[
 \begin{cases}
  \nu([x,\infty)) \le Cx^{-\beta}, & \beta > 0\\
  \nu([1,x]) \le Cx^{-\beta}, & \beta \le 0
 \end{cases}
\]
In particular, $\sum_{n\in\N^*} \nu(n) |z|^n < \infty$ for every $|z| < 1$. Moreover, $\nu$ is finite if $\alpha > 0$.
\end{lemma}
\begin{proof}
 Fix $a < M < A$ and $\ep>0$. Let $\nu$ be as in the statement and recall the definition of the transition matrix $P$.
 From the branching property and the law of large numbers we get for large $n$,
 \[
  \delta_n P ([n/A,n/a)) = \P_n(Z_1 \in[n/A,n/a)) \ge 1-\ep.
 \]
 This implies that for every $x$ large enough and $y\ge x$,
 \begin{equation*}
  \nu P([x/A,y/a)) \ge (1-\ep)\nu([x,y)).
 \end{equation*}
Now let $\ep \to 0$. The previous inequality together with \eqref{eq:nu_P} (with $\lambda = m^\alpha$) gives for every $\beta < \alpha$, for every $x$ large enough and $y \ge x$,
 \begin{equation}
 \label{eq:nu1}
  \nu([x,y)) \le m^\beta \nu ([x/A,y/a)).
 \end{equation}
 
 Now set $b_n = \nu([M^n,M^{n+1}))$. Iterating \eqref{eq:nu1} and choosing $A$ and $a$ close to $M$ one readily shows that for every $\beta < \alpha$ and $\delta > 0$, there exists $K\in\N$ such that for $n\ge K$,
 \[
  b_{K+n} \le m^{\beta n}(b_K + \cdots + b_{K + \lfloor \delta n\rfloor + 1}).
 \]
Elementary arguments yield the first statement of the lemma. The remaining statements follow.
\end{proof}

\begin{lemma}
 \label{lem:Lambda}
 Let $f:(0,\infty)\to(0,\infty)$ be measurable and satisfying for some constant $C\ge1$:
 \[
  \forall n\in\Z: \text{ either $f \equiv 0$ on $[m^n,m^{n-1})$ or }\forall x,y\in [m^n,m^{n-1}): f(x)/f(y) \in [C^{-1},C].
 \]
Furthermore, let $\Lambda$ be a measure on $(0,\infty)$ satisfying $\Lambda(A) = \Lambda(mA)$ for all Borel $A\subset(0,\infty)$ and $\Lambda([m,1)) \ne 0$. Then
 \[
 \frac 1 {\Lambda([m,1))}\int_0^\infty f(x)\,\Lambda(dx) \asymp_C \frac 1 {\log m^{-1}}\int_0^\infty f(x)\,\frac{dx}{x},
 \]
where we set $a\asymp_C b \iff C^{-1} b\le a\le Cb$.
\end{lemma}
\begin{proof}
First note that the restriction of the measure $\Lambda([m,1))^{-1}\Lambda$ to $[m,1)$ can be written as the image of the measure $(\log m^{-1})^{-1}\frac{dx}{x}$ on $[m,1)$ under a suitable map $\widetilde\varphi$: first map the latter via its distribution function to Lebesgue measure on $[0,1]$, then map this back to $[m,1)$ via the inverse of the distribution function of the measure $\Lambda([m,1))^{-1}\Lambda$. Then extend the map $\widetilde\varphi$ to a map $\varphi$ on $(0,\infty)$ by
\[
\varphi(x) = m^n \widetilde\varphi(m^{-n}x),\quad x\in[m^{n+1},m^n),\ n\in\Z. 
\]
By the self-similarity of the measures $\Lambda$ and $\frac{dx}{x}$, the measure $\Lambda([m,1))^{-1}\Lambda$ is indeed the  image of the measure $(\log m^{-1})^{-1}\frac{dx}{x}$ on $(0,\infty)$ by the map $\varphi$. Furthermore, by construction the map $\varphi$ maps every interval $[m^n,m^{n-1})$, $n\in\Z$, to itself. In particular, for all $x>0$, either $f(\varphi(x)) = f(x) = 0$ or $f(\varphi(x))/f(x)\in[C^{-1},C]$ by assumption. The lemma easily follows by the change of variables formula.
\end{proof}

\begin{corollary}
\label{cor:mu}
 Let $\mu$ be a non-zero measure on $(0,\infty)$ satisfying for some $\alpha\in\R$, $\mu(A) = m^{-\alpha}\mu(mA)$ for all Borel $A\subset(0,\infty)$. Then for every $z\in(0,1)$,
 \[
     \int_0^\infty (e^{(H(z)-1)x}-e^{-x})\,\mu(dx) < \infty \iff \alpha < 1.
 \]
\end{corollary}
\begin{proof}
Note that $\mu([m,1)) \ne 0$, otherwise we would have $\mu = 0$ by self-similarity. Define $\Lambda(dx) = x^\alpha \mu(dx)$. Then $\Lambda$ satisfies the hypothesis of Lemma~\ref{lem:Lambda}. Now let $\beta,\gamma \in\R\cup\{+\infty\}$ and set $$f_{\beta,\gamma}(x) = x^\beta\Ind_{(x<1)} + x^{-\gamma}\Ind_{(x>1)},\quad x> 0.$$ By Lemma~\ref{lem:Lambda} applied to the function $x\mapsto x^{-\alpha} f_{\beta,\gamma}(x)$, we have for some $C>1$ (depending on $\Lambda$, $m$, $\alpha$, $\beta$ and $\gamma$),
\[
  \int_0^\infty f_{\beta,\gamma}(x)\,\mu(dx) =  \int_0^\infty x^{-\alpha}f_{\beta,\gamma}(x)\,\Lambda(dx) \asymp_C \int_0^\infty x^{-\alpha-1} f_{\beta,\gamma}(x)\,dx.
\]
In particular, this shows that
\begin{equation}
 \label{eq:finiteness}
 \int_0^\infty f_{\beta,\gamma}(x)\,\mu(dx) < \infty \iff \beta > \alpha\tand \gamma > \alpha,
\end{equation}
with the obvious meaning if $\beta = +\infty$ or $\gamma = +\infty$.

Let $z\in(0,1)$, so that $H(z)\in(0,1)$. We finally consider the integral
\begin{equation}
 \label{eq:finiteness2}
   \int_0^\infty (e^{(H(z)-1)x}-e^{-x})\,\mu(dx).
\end{equation}
For large $x$, the integrand is smaller than any fixed polynomial, so that the integral always converges at $\infty$ by \eqref{eq:finiteness} applied with $\beta=+\infty$ and some $\gamma > \alpha$. On the other hand, as $x\to0$, the integrand is asymptotically equivalent to $H(z) x$. Equation~\eqref{eq:finiteness} applied with $\beta = 1$ and $\gamma = +\infty$ then implies that the integral in \eqref{eq:finiteness2} converges at the origin if and only if $\alpha < 1$. These two facts prove the corollary.
\end{proof}

\begin{proof}[Proof of Theorem~\ref{th:main}]
Although we could restrict ourselves to the pure death case, i.e.\ $F(z) = 1-m(1-z)$ (see Section~\ref{sec:history}), we prove the theorem immediately in its generality.

We first introduce some notation. Let $Y_n$ denote a random variable with the law of $Z_n$ under $\P_1(\cdot\,|\,Z_n > 0)$. By \eqref{eq:yaglom}, $Y_n$ converges in law to the Yaglom distribution $\nu_{\text{min}}$, in particular, $H_n(z) = \E[z^{Y_n}] \to H(z)$ as $n\to\infty$. Note that the inverse $H^{-1}$ exists on $[0,1]$ and is continuous. We further define $p_n = \P_1(Z_n > 0)$ for $n\in\N$ and note that $p_{n+1}/p_n\to m$ as $n\to\infty$ by \eqref{eq:pn}.

Now let $\nu$ be an $m^\alpha$-invariant measure, $\alpha\in\R$. Denote by $G(z) = \sum_{n=1}^\infty \nu(n)z^n$ its generating function, which is finite and well-defined for $|z|<1$ by Lemma~\ref{lem:tail}. We will extend the notation $\P_x$ and $\E_x$ to the (possibly infinite) measure $\nu$ by $\P_\nu(\cdot) = \sum_n \nu(n)\P_n(\cdot)$ and $\E_\nu[\cdot] = \sum_n \nu(n)\E_n[\cdot]$. 

Define the random variable $N_n$ to be the number of individuals at time $0$ which have a descendant at time $n$. Then $N_n > 0$ iff $Z_n > 0$. Furthermore, by the branching property, $Z_n$ is equal in law (under $\P_k$ for every $k\in\N^*$) to $Y^{(1)}_n+\cdots+Y^{(N_n)}_n$, where the variables $Y^{(i)}_n$ are iid copies of $Y_n$ and independent of $N_n$. Hence, as $n\to\infty$, by the $m^\alpha$-stationarity of $\nu$,
\begin{align}
\nonumber
m^{-\alpha n}\E_\nu[z^{N_n}\Ind_{N_n > 0}] &= m^{-\alpha n} \E_\nu[H_n^{-1}(z)^{Z_n}\Ind_{Z_n > 0}] = \E_\nu[H_n^{-1}(z)^{Z_0}]\\
\label{eq:apple}
&=  G(H_n^{-1}(z))
\to G(H^{-1}(z)),\quad \tas n\to\infty.
\end{align}

Now note that under $\P_k$, $N_n$ is binomially distributed with parameters $k$ and $p_n$ for every $k\in\N^*$.
In particular,
\begin{align}
\label{eq:didi}
\E_\nu[z^{N_n}\Ind_{N_n > 0}] = \E_\nu[z^{N_n} - 0^{N_n}] = \E_\nu[(1-p_n(1-z))^{Z_0} - (1-p_n)^{Z_0}].
\end{align}
Defining for every $n\in\N$ the measure $\mu_n$ by $\mu_n(A) = m^{-\alpha n}\nu(p_n^{-1}A)$ for Borel $A\subset(0,\infty)$, we thus get by \eqref{eq:apple} and \eqref{eq:didi},
\begin{equation}
\label{eq:zN}
 G(H^{-1}(z)) = \lim_{n\to\infty} \int_0^\infty \big((1-p_n(1-z))^{x/p_n} - (1-p_n)^{x/p_n}\big)\,\mu_n(dx).
\end{equation}
With the first inequality in Lemma~\ref{lem:ineq} and the fact that $p_n\to0$ as $n\to\infty$, this gives
\begin{equation}
 \label{eq:domination}
 \forall w\in(0,1):\quad\sup_n \int_0^\infty \big(e^{(w-1)x}-e^{-x}\big)\,\mu_n(dx) < \infty.
\end{equation}
Using \eqref{eq:domination} with $w=1/2$, say, gives that the sequence of measures $\widetilde\mu_n(dx) = xe^{-x}\mu_n(dx)$ is tight and therefore, by Prokhorov's theorem, precompact in the space of finite measures on $[0,\infty)$ endowed with weak convergence.
Let $\widetilde\mu$ be a subsequential limit and define the measure $\mu(dx) = x^{-1}e^x\widetilde\mu_{(0,\infty)}(dx)$, where $\widetilde\mu_{(0,\infty)}$ is the restriction of $\widetilde\mu$ to $(0,\infty)$. We claim that
\begin{equation}
 \label{eq:limit}
 G(z) = \int_0^\infty (e^{(H(z)-1)x}-e^{-x})\,\mu(dx) + \widetilde\mu(0) H(z).
\end{equation}
Indeed, fix $z\in(0,1)$ and denote by $g_{z,n}(x)$ the integrand on the right-hand side of \eqref{eq:zN}. Then the function $x\mapsto g_{z,n}(x)/(xe^{-x})$,  continuously extended to $[0,\infty)$, converges uniformly on every compact subset of $[0,\infty)$ to the function $\widetilde g_z$ defined by $\widetilde g_z(x) = (e^{zx}-1)/x$ for $x>0$ and $\widetilde g_z(0) = z$. Using \eqref{eq:domination} with some $w\in(z,1)$ together with the second inequality in Lemma~\ref{lem:ineq}, a truncation argument then shows that we can pass to the (subsequential) limit inside the integral in \eqref{eq:zN}, which yields 
\[
 G(H^{-1}(z)) = \int_0^\infty \widetilde g_z(x)\,\widetilde\mu(dx).
\]
This yields \eqref{eq:limit}. Furthermore, the theory of Laplace transforms gives that $\mu$ and $\widetilde\mu(0)$, hence $\widetilde\mu$, are uniquely determined by \eqref{eq:limit}, so that $\widetilde\mu_n$ converges in fact weakly to $\widetilde\mu$. As a consequence, $\mu_n$ converges vaguely on $(0,\infty)$ to $\mu$.

The scaling properties of the measure $\mu$ follow from this convergence: we have for every compact interval $A\subset(0,\infty)$ whose endpoints are not atoms of $\mu$,
\[
\mu(A) = \lim_{n\to\infty} m^{-\alpha (n+1)} \nu(p_{n+1}^{-1}A) = m^{-\alpha} \lim_{n\to\infty} m^{-\alpha n} \nu(p_n^{-1}(p_n/p_{n+1})A) = m^{-\alpha} \mu(m^{-1}A),
\]
since $p_{n+1}/p_n\to m$ by \eqref{eq:pn}. This implies that the measure $\Lambda$ defined by $\Lambda(dx) = x^\alpha \mu(dx)$ satisfies $\Lambda(A) = \Lambda(mA)$ for every Borel set $A$. 

It remains to investigate which terms in \eqref{eq:limit} vanish for particular values of $\alpha$. A first constraint comes from the fact that $G(z)$ is finite for every $z\in(0,1)$ by Lemma~\ref{lem:tail}, and so the integral in \eqref{eq:limit} needs to be finite as well. By Corollary~\ref{cor:mu}, this is true if and only if $\alpha < 1$ or $\mu = 0$.

A second constraint comes from the fact that $G$ satisfies \eqref{eq:gf_measure} with $\lambda = m^\alpha$. To verify this, we first recall the following equations for the function $H$:
\begin{align}
\label{eq:H1}
  H(F(z)) - H(F(0)) &= m H(z),\quad |z|\le 1.\\
  \label{eq:H2}
  H(F(0)) &= 1-m\\
  \label{eq:H3}
  H(F(z)) - 1 &= m(H(z) - 1),\quad |z|\le 1.
\end{align}
Indeed, \eqref{eq:H1} is an immediate consequence of \eqref{eq:gf_measure} (with $\lambda = m$) and the finiteness of $H$ for $|z| = 1$, \eqref{eq:H2} follows from \eqref{eq:H1} by setting $z=1$, and \eqref{eq:H3} follows from \eqref{eq:H1} and \eqref{eq:H2} by reordering terms. We now have by \eqref{eq:limit}, for every $z\in(0,1)$,
\begin{align*}
 &G(F(z)) - G(F(0)) \\
 &= \int_0^\infty [e^{(H(F(z)) - 1)x} - e^{-x} - e^{(H(F(0)) - 1)x} + e^{-x}]\,\mu(dx) + \widetilde\mu(0) (H(F(z)) - H(F(0)))\\
 &= \int_0^\infty [e^{m(H(z) - 1)x} - e^{mx}]\,\mu(dx) + m\widetilde\mu(0) H(z)\qquad \text{(by \eqref{eq:H3},\eqref{eq:H2},\eqref{eq:H1} (in this order))}\\
 &= m^\alpha \int_0^\infty [e^{(H(z) - 1)x} - e^{x}]\,\mu(dx) + m\widetilde\mu(0) H(z)\qquad  \text{(by self-similarity of $\mu$)}.
\end{align*}

Comparing with \eqref{eq:gf_measure}, this implies that $\widetilde\mu(0) = 0$ unless $\alpha = 1$. Summing up, we have the following constraints for the quantities in \eqref{eq:limit}:
\begin{itemize}
 \item $\alpha > 1$: $\mu = 0$ and $\widetilde\mu(0) = 0$
 \item $\alpha = 1$: $\mu = 0$
 \item $\alpha < 1$: $\widetilde\mu(0) = 0$. 
\end{itemize}
This proves the necessity part of the theorem.

For the sufficiency, we only need to consider the case $\alpha < 1$. Let $G$ be a function given by \eqref{eq:G_rep} with $\Lambda$ a measure on $(0,\infty)$ satisfying $\Lambda(A) = \Lambda(mA)$ for every Borel set $A$. By the above calculations, one readily shows that $G$ satisfies \eqref{eq:gf_measure} with $\lambda = m^\alpha$. It remains to show that $G$ is the generating function of a (locally finite) measure on $\N^*$. Now, for every $x\in(0,\infty)$, the function $z\mapsto e^{(H(z)-1)x} - e^{-x}$ is the generating function of the sum of $\mathcal N$ iid random variables distributed according to $\nu_{\text{min}}$, where $\mathcal N \sim \Poi(x)$, restricted on the event that this sum is positive (see also Section~\ref{sec:subordinators}). Hence, $G$ is an integral over a family of generating functions and thus the generating function of a (not necessarily locally finite) measure. But by Corollary~\ref{cor:mu}, $G(z)$ is finite for $z\in(0,1)$, so that this measure is indeed locally finite. This finishes the proof of the sufficiency part of the theorem.
\end{proof}

\begin{proof}[Proof of Theorem~\ref{th:qsd}]
 Let $\nu$ be a non-trivial $m^\alpha$-invariant measure of the BGW process $Z$, $\alpha \le 1$. By Theorem~\ref{th:main}, it remains to show that $\nu$ is finite if and only if $\alpha\in(0,1]$. For $\alpha = 1$ this is immediate, suppose therefore that $\alpha<1$. Denote by $G$ the generating function of the measure $\nu$ and let $\Lambda$ be the measure from Theorem~\ref{th:main}. Then Lemma~\ref{lem:Lambda} easily implies that the integral $\int_0^\infty (1-e^{-x})x^{-\alpha}\,\Lambda(dx)$ converges at the origin for all $\alpha < 1$ but converges at $\infty$ if and only if $\alpha \in(0,1]$. Hence, $\|\nu\| = G(1) < \infty$ if and only if $\alpha \in (0,1]$. This proves the theorem.
\end{proof}

\begin{proof}[Proof of Theorem~\ref{th:true_main}]
The first two parts are known, see the discussion before the statement of the theorem. The third part can be proven by adapting the proof of Theorem~\ref{th:main}. Alternatively, it can be derived from Theorem~\ref{th:main} as follows: Let $\lambda = m^\alpha > 1$ (hence, $\alpha < 0$). Let $\nu$ be a measure on $\N$ and denote by $\nu^*$ its restriction to $\N^*$. Denote by $G$ and $G^*$ the generating functions of $\nu$ and $\nu^*$, respectively, note that $G^* = G - G(0)$. Since $0$ is an absorbing state for the process $Z$, the measure $\nu$ is a true $\lambda$-invariant measure for $Z$ if and only if the following two statements hold:
\begin{enumerate}
	\item $\nu^*$ is a $\lambda$-invariant measure for $Z$.
	\item $\nu P(0) = \lambda \nu(0)$, equivalently, $G(F(0)) = \lambda G(0)$.
\end{enumerate}
By Theorem~\ref{th:main}, the first statement is equivalent to
\[
G(z) = \int_0^\infty e^{(H(z)-1)x}\frac 1 {x^\alpha}\,\Lambda(dx) + C,
\]
for some constant $C$ and $\Lambda$ as in the statement of Theorem~\ref{th:main} (it can easily be seen that the integral converges using Lemma~\ref{lem:Lambda}, as in the proof of Corollary~\ref{cor:mu}). Together with \eqref{eq:H2} and the self-similarity of $\Lambda$, this gives
\[
G(F(0)) = \int_0^\infty e^{-mx}\frac 1 {x^\alpha}\,\Lambda(dx) + C = \lambda \int_0^\infty e^{-x}\frac 1 {x^\alpha}\,\Lambda(dx) + C = \lambda G(0) + C.
\]
Hence, given the first statement, the second statement is equivalent to $C=0$, which proves the theorem.
\end{proof}

\section{Discussion}
\label{sec:discussion}

\subsection{The Kesten--Spitzer formula for invariant measures and the minimal Martin entrance boundary}
\label{sec:martin}

To our knowledge, Theorem~\ref{th:main}, and more specifically Formula \eqref{eq:G_rep}, was previously known only for $\lambda = 1$ (i.e., $\alpha=0$). In this case, one simply says \emph{invariant} instead of $1$-invariant. 
In the literature (Kesten--Spitzer \cite{Spitzer1967}, Athreya--Ney~\cite[p.\,69]{Athreya1972}, Hoppe \cite{Hoppe1977}; see Section~\ref{sec:history} below for the history of the result), one generally finds this result under the following form: A function $Q(z)$ is the generating function of an invariant measure for the BGW process $Z$ if and only if  there exists a constant $c\ge0$ and a probability measure $\mu$ on $[0,1)$, such that 
\begin{equation}
\label{eq:Q_rep2}
Q(z) = c \int_0^1 \sum_{n=-\infty}^\infty [\exp((H(z)-1)m^{n-t}) - \exp(-m^{n-t})]\,\mu(dt).
\end{equation}
This is the Choquet decomposition of $Q(z)$ as a convex combination of generating functions of extremal invariant measures.
One easily sees that \eqref{eq:G_rep} (with $\alpha=0$) and \eqref{eq:Q_rep2} are equivalent: Given $c\ge0$ and a measure $\mu$ such that \eqref{eq:Q_rep2} holds, we can define a measure $\Lambda$ on $[1,m^{-1})$ as the push-forward of the measure $c\mu$ by the map $t\mapsto m^{-t}$. The measure $\Lambda$ can then be uniquely extended to $(0,\infty)$ in such a way that $\Lambda(A) = \Lambda(mA)$ for every Borel set $A$. One easily checks that \eqref{eq:G_rep} holds with this $\Lambda$ and $\alpha=0$. Conversely, given such a measure $\Lambda$, one can define a finite measure $\widetilde\mu$ on $[0,1)$ as the push-forward of the measure $\Lambda(\cdot \cap [1,m^{-1}))$ by the inverse map $x\mapsto \log_{m^{-1}} x$. Setting $c=\widetilde\mu([0,1))$ and $\mu = \widetilde\mu/c$ gives \eqref{eq:Q_rep2}.

We now relate formula \eqref{eq:Q_rep2} to Martin boundary theory, see \cite[Chapter~10]{Kemeny1976} for an introduction to this theory\footnote{Another very good and more modern introduction is \cite[Chapter~IV]{Woess2000}. He only considers Martin \emph{exit} boundaries but one can reduce to this case in our setting by considering the transition matrix $\hat P_{ij} = \nu_{\text{min}}(j)P_{ji}/\nu_{\text{min}}(i)$ instead of $P$.}.
We briefly recall the basic constructions of interest to us. Let $P$ be the transition matrix of a transient sub-Markov chain on $\N^*$. Define the Green kernel $G(x,y) = \sum_{k=0}^\infty (P^k)_{xy}$ and assume there is a state $o\in\N^*$ such that $G(x,o) > 0$ for all $x\in\N^*$ (in the case of subcritical BGW processes killed at $0$, we choose $o$ to be the span of the reproduction law). This allows to define the Martin kernel by $K(x,y) = G(x,y)/G(x,o)$. The \emph{Martin entrance compactification} of $\N^*$ is then defined as the smallest compactification $\mathcal M$ of the discrete set $\N^*$ such that all measures $K(x,\cdot)$ extend continuously (w.r.t.\  pointwise convergence of measures seen as functions on $\N^*$).  Every point $\xi$ on the \emph{Martin entrance boundary} $B = \mathcal M\backslash \N^*$ thus defines an invariant measure $K(\xi,\cdot)$ with mass 1 at $o$. Moreover, every \emph{extremal} invariant measure, meaning that it can not be written as a non-trivial convex combination of invariant measures, arises this way. The set of those points $\xi\in B$ for which $K(\xi,\cdot)$ is extremal is called the \emph{minimal} Martin entrance boundary, denoted by $B^*$. The \emph{Poisson-Martin integral formula} now assigns to every invariant measure $\nu$ a unique integral representation in terms of extremal invariant measures, namely,
\[
\nu = \int_{B^*} K(\xi,\cdot)\,\mu^\nu(d\xi),
\]
for a finite measure $\mu^\nu$ on $B^*$. 

The construction outlined in the previous paragraph is the approach used by Kesten and Spitzer \cite{Spitzer1967} to derive formula \eqref{eq:Q_rep2} (for pure death processes). In particular, their proof implies that the extremal invariant measures of a subcritical BGW process are (up to multiplicative constants) the measures $\nu_t$, $t\in S^1 = [0,1]_{0\sim 1}$, with generating functions
\[
\sum_{k=1}^\infty \nu_t(k) z^k = \sum_{n=-\infty}^\infty [\exp((H(z)-1)m^{n-t}) - \exp(-m^{n-t})].
\]
Defining $\xi_t\in B^*$ by $\nu_t = K(\xi_t,\cdot)$, the map $t\mapsto \xi_t$ is thus (by extremality) a bijection between the compact space $S^1$ and $B^*$, moreover, one easily sees that it is continuous. \emph{It follows that the minimal Martin entrance boundary $B^*$ is homeomorphic to the circle $S^1$.}

Now let $\lambda > 0$. The above construction can be performed with the operator $P/\lambda$ instead of $P$, giving rise to a $\lambda$-boundary theory for all $\lambda$ such that the $\lambda$-Green function $G_\lambda = \sum_{k=0}^\infty \lambda^{-k}(P^k)_{xy}$ is finite. The infimum of these values of $\lambda$ is the \emph{spectral radius} $\rho = \lim_{k\to\infty} ((P^k)_{oo})^{1/k}$ \cite[Chapter~II]{Woess2000}, which equals $m$ for subcritical BGW processes by \eqref{eq:yaglom} and \eqref{eq:pn}. For $\lambda > \rho = m$, Theorem~\ref{th:main} then implies a formula similar to \eqref{eq:Q_rep2}. A reasoning as  in the last paragraph yields the following:
\begin{corollary}
\label{cor:stable}
For every $\lambda > m$, the minimal $\lambda$-Martin entrance boundary of the BGW process $(Z_n)_{n\ge 0}$ is homeomorphic to the circle $S^1$.
\end{corollary}

In particular, Corollary \ref{cor:stable} shows that all minimal $\lambda$-Martin entrance boundaries, $\lambda > m$, are homeomorphic. This remarkable fact is part of a property called \emph{stability} by some authors \cite[p301]{Woess2000} and holds true for example for the (exit) boundary of random walks on trees and hyperbolic graphs. We know of no general theory that yields this result without explicitly calculating the $\lambda$-Martin entrance boundaries for every $\lambda$.

We finish this section with a discussion of the case $\lambda = m$, for which Theorem~\ref{th:main} gives that the minimal $m$-Martin entrance boundary is trivial, i.e. there exists up to multiplicative constants only one $m$-invariant measure. This fact is quite common and holds in general for example if the process is $m$-recurrent, i.e. if $G_m(x,y) = \infty$ for all (some) $x,y$ \cite[Chapter~IV]{Woess2000}. Note that $m$-recurrence is equivalent to recurrence of the so-called $Q$-process, which is in our case the Markov process with transition matrix $Q$ given by $Q_{ij} = m^{-1}P_{ij} j/i$ (recall that the function $h(i) = i$ is $m$-harmonic for our process, i.e. $Ph = mh$). It is remarkable that in our setting the $Q$-process may be positive recurrent, null recurrent or transient. This fact does not seem to appear in the usually cited  monographs on branching processes\footnote{It was even claimed in the literature that recurrence always holds \cite[p972]{Pakes1999}.}, only a criterion for positive recurrence is easy to find (see e.g. \cite[p59]{Athreya1972}): the $Q$-process is positive recurrent if and only if $\E_1[Z_1\log Z_1] < \infty$. However, Joffe proved in 1967 already the following recurrence criterion \cite{Joffe1967}: Let $F(z)$ denote the generating function of the offspring distribution and define $\eta$ by $1-F(z) = m(1-z)(1-\eta(z))$. Set $q_n = \P_1(Z_n = 0)$. Then the $Q$-process is recurrent if and only if the following sum diverges:
\[
 \sum_{n=1}^\infty\prod_{k=1}^{n} (1-\eta(q_k)).
\]
Since $1-q_n = m^{n+o(n)}$ by \eqref{eq:pn}, one can easily construct examples where the above sum converges (so that the $Q$-process is transient), for example when $\eta(z)\ge 1/|\log(1-z)|^\beta$ for some $\beta < 1$ and $z$ close to $1$.


\subsection{Probabilistic interpretation of \texorpdfstring{\eqref{eq:G_rep}}{(\ref{eq:G_rep})} and relation with semi-stable subordinators}
\label{sec:subordinators}

Let $\nu$ be a QSD of eigenvalue $m^\alpha$ of the BGW process, $\alpha\in(0,1)$. By Theorem~\ref{th:qsd}, it admits the representation \eqref{eq:G_rep} with a measure $\Lambda$ as in the statement of the theorem. Let $\mathcal N$ be a random variable whose generating function is equal to the right-hand side of \eqref{eq:G_rep}, but with $H(z) \equiv z$. Then $\nu$ is the law of the sum of $\mathcal N$ iid random variables distributed according to the Yaglom distribution $\nu_{\text{min}}$. As for the law of $\mathcal N$, 
expanding the exponential in \eqref{eq:G_rep} gives
\begin{equation}
\label{eq:N}
\forall k\ge 1: \P(\mathcal N = k) = \int_0^\infty e^{-x}\frac{x^k}{k!}\frac 1 {x^\alpha}\,\Lambda(dx),\quad \P(\mathcal N = 0) = 0.
\end{equation}
Heuristically, $\mathcal N$ is therefore a Poisson-distributed random variable with a random parameter drawn according to the measure $x^{-\alpha}\,\Lambda(dx)$ and conditioned to be non-zero. A way to make this rigorous (note that the measure $x^{-\alpha}\Lambda(dx)$ has infinite mass!) is using \emph{subordinators}, of which we first recall the basic facts.

A subordinator $S = (S_t)_{t\ge 0}$ is a real-valued, non-decreasing process with stationary and independent increments. We always assume $S_0 = 0$. Then the law of $S$ is determined by its \emph{cumulant} $\kappa_S(\theta) = -\log \E[e^{-\theta S_1}]$, which satisfies the \emph{L\'evy--Khintchine} formula (see e.g.\ \cite[Ch.\,13]{Kallenberg1997} or \cite{Bertoin1996}),
\begin{equation}
\label{eq:kappa}
 \kappa_S(\theta) = a\theta + \int_0^\infty (1-e^{-\theta x})M(dx),
\end{equation}
where $a\ge0$ is called the \emph{drift} and $M$ is a measure on $(0,\infty)$ called the \emph{L\'evy measure} of the subordinator and satisfying $\int_0^\infty (1-e^{-x}) M(dx) < \infty$.

If $T = (T_t)_{t\ge0}$ is another subordinator (or, in general, a L\'evy process) independent of $(S_t)_{t\ge0}$, then the \emph{subordinated process} $T\circ S := (T_{S_t})_{t\ge0}$ is again a subordinator (L\'evy process) with cumulant
\begin{equation}
\label{eq:kappa_sub}
 \kappa_{T\circ S} = \kappa_S\circ\kappa_T.
\end{equation}
If $N = (N_t)_{t\ge0}$ is a driftless subordinator whose L\'evy measure is a probability measure on $\N^*$, then the subordinators $N$ and $N\circ S$ both take values in $\N^*$ and $N\circ S$ is therefore again a driftless subordinator with L\'evy measure concentrated on $\N^*$. Let $H$ and $G$ denote the generating functions of the L\'evy measures of $N$ and $N\circ S$, respectively. Note that $H(1) = 1$ and $G(1)<\infty$, because a L\'evy measure on $\N^*$ is necessarily finite. It follows from \eqref{eq:kappa} (applied first to $N\circ S$ and then to $N$) and \eqref{eq:kappa_sub} that
\begin{equation}
 \label{eq:GH}
 G(1) - G(z) = \kappa_{N\circ S}(-\log z) = \kappa_S(1-H(z)).
\end{equation}
Setting $z=0$ yields $G(1) = \kappa_S(1)$. Rearranging \eqref{eq:GH}, we get with \eqref{eq:kappa},
\begin{equation}
 \label{eq:GH2}
 G(z) = \kappa_S(1) - \kappa_S(1-H(z)) = aH(z) + \int_0^\infty (e^{(H(z)-1)x}-e^{-x})\,M(dx).
\end{equation}

We apply the previous equations to the QSD $\nu$, by setting $a = 0$ and $M(dx) = \frac 1 {x^\alpha}\,\Lambda(dx)$, note that $\kappa_S(1) = \int_0^\infty (1-e^{-x})\,M(dx) = 1$. In particular, $M$ is a L\'evy measure, so that the subordinator $S$ is well defined. We also let the L\'evy measure of the subordinator $N$ be $\nu_{\text{min}}$; we recall that its generating function is indeed denoted by $H(z)$. Equation \eqref{eq:GH2} then gives a probabilistic interpretation to the QSD $\nu$: it says that \emph{$\nu$ is the L\'evy measure of the subordinator $N\circ S$ (or, equivalently, the law of its first jump)}.

Note that the case $\alpha = 1$ may also be covered by setting $a = 1$ and $M \equiv 0$ in \eqref{eq:kappa}, i.e. taking the subordinator $S = \mathrm{Id}$.

This fact allows for an alternative statement of Theorem~\ref{th:qsd}. For this, we introduce the notion of a \emph{semi-stable} subordinator: we say that the subordinator $S = (S_t)_{t\ge0}$ is \emph{$(\alpha,m)$-semi-stable}\footnote{This terminology is taken from \cite[Section~9.2]{Embrechts2002}.}, if
\begin{equation}
\label{eq:semistable}
 (S_{m^\alpha t})_{t\ge0} \stackrel{\text{law}}{=} (m S_t)_{t\ge0},
\end{equation}
or, in terms of the cumulant,
\begin{equation}
 \label{eq:ss_cumulant}
 \kappa_S(m \theta) = m^\alpha \kappa_S(\theta),\quad \theta\ge0.
\end{equation}
One easily obtains from \eqref{eq:kappa} and \eqref{eq:ss_cumulant} the following characterization of semi-stable subordinators: A subordinator $S = (S_t)_{t\ge0}$ with drift $a$, L\'evy measure $M$, satisfying $S_0 = 0$ and $S \not\equiv 0$, is $(\alpha,m)$-semi-stable, $\alpha\in\R$, $m\in(0,1)$, if and only if
\begin{itemize}
 \item $\alpha\in(0,1)$, $a = 0$ and $M(dx) = \frac 1 {x^\alpha}\,\Lambda(dx)$ for a measure $\Lambda$ on $(0,\infty)$ satisfying $\Lambda(A) = \Lambda(mA)$ for all Borel $A\subset (0,\infty)$, \emph{or}
 \item $\alpha = 1$, $a > 0$ and $M\equiv 0$.
\end{itemize}

The previous arguments then give the following equivalent statement of Theorem~\ref{th:qsd}:
\begin{theorem}
 \label{th:qsd2}
 The quasi-stationary distributions of eigenvalue $m^\alpha$ of the BGW process, $\alpha\in\R$, are exactly the L\'evy measures of the subordinators $N \circ S$, where $N$ is the driftless subordinator with L\'evy measure $\nu_{\text{min}}$ and $S$ is an $(\alpha,m)$-semi-stable subordinator with $\kappa_S(1) = 1$.
\end{theorem}

\begin{remark}
 One can drop the requirement $\kappa_S(1) = 1$ in the above theorem if one replaces ``are exactly the L\'evy measures'' by ``are exactly the laws of the first jumps''.
\end{remark}

\paragraph{Composition of generating functions}

 Let $G_\alpha$ be the generating function of a QSD of $Z$ with eigenvalue $m^\alpha$, $\alpha\in(0,1]$. Furthermore, let $G_\beta$ be the generation function of an $m^{\alpha\beta}$-invariant measure, $\beta\le 1$, of the pure death process with mean offspring $m^\alpha$, i.e. with $F(z) = 1-m^\alpha(1-z)$. It is easy to see from \eqref{eq:gf_measure} that the composition $G_\beta\circ G_\alpha$ is the generating function of an $m^{\alpha\beta}$-invariant measure of $Z$ (note that the Yaglom distribution of a pure death process is always $\delta_1$, hence its generating function is the identity $z\mapsto z$). If $\Lambda_\alpha$, $\Lambda_\beta$ and $\Lambda_{\alpha\beta}$ are the measures from Theorem~\ref{th:main} corresponding to $G_\alpha$, $G_\beta$ and $G_\beta\circ G_\alpha$, respectively, then one may ask the following question:
 
 \begin{question}
 \label{q:lambda}
  Is there a simple formula expressing $\Lambda_{\alpha\beta}$ in terms of $\Lambda_\alpha$ and $\Lambda_\beta$?
 \end{question}
We were not able to answer this question and are in fact doubtful that the answer is positive in general. In order to rephrase this problem into a more familiar setting, consider the case where $G_\beta$ is the generating function of a probability measure, so that in particular $\beta\in(0,1]$. Let $S^\alpha$ and $S^\beta$ be the $(\alpha,m)$- and $(\beta,m^\alpha)$-semi-stable subordinators associated to $G_\alpha$ and $G_\beta$ by Theorem~\ref{th:qsd2}. In particular, $\kappa_{S^\alpha}(1) = \kappa_{S^\beta}(1) = 1$. By \eqref{eq:GH} and \eqref{eq:kappa_sub}, we then have
\[
 1- G_\beta\circ G_\alpha(z) = \kappa_{S^\beta}(1-G_\alpha(z)) = \kappa_{S^\beta}(\kappa_{S^\alpha}(1-H(z))) = \kappa_{S^\alpha\circ S^\beta}(1-H(z)).
\]
Hence, Question~\ref{q:lambda} is equivalent to the question of whether there is a simple formula expressing the L\'evy measure of $S^\alpha\circ S^\beta$ in terms of the L\'evy measures of $S^\alpha$ and $S^\beta$. To the best of our knowledge, no such formula is known, and, given the fact that the Laplace transform of a measure has no simple inversion formula, there does not seem to be much hope.

 \subsection{History of the problem}
\label{sec:history}

The study of $\lambda$-invariant measures of subcritical BGW processes has a rich history which we aim to elucidate here. The starting point seems to be Yaglom's 1947 article  \cite{Yaglom1947}, who showed the existence of the now-called Yaglom limit of a subcritical BGW process under the assumption of finite variance\footnote{The assumption of finite variance was later removed in \cite{Heathcote1967,Joffe1967}.}. The BGW process appeared again  as an important example in the seminal paper by Seneta and Vere-Jones \cite{Seneta1966} on QSD of Markov processes on (countably) infinite state spaces. In this work, the authors show that subcritical BGW processes admit a one-parameter family of QSD whose generating functions are $1-(1-H(z))^\alpha$, $\alpha\in(0,1]$, with $H(z)$ denoting, as above, the generating function of the Yaglom limit. Rubin and Vere-Jones \cite{Rubin1968} raised the question whether there existed other QSD. They failed to answer the question in general but showed that these QSD where the only ones with regularly varying tails.

These works on QSD of subcritical BGW process seem to have been independent of other works on ($1$-)invariant measures: In 1965, Kingman \cite{Kingman1965} showed that invariant measures for a subcritical  BGW process are not unique, which, as claimed by Kingman, disproved a conjecture by Harris. A full characterization of invariant measures, Formula \eqref{eq:Q_rep2}, was then given by Kesten and Spitzer in 1967 \cite{Spitzer1967} (they also gave credit to H.~Dinges for deriving the formula independently), motivated by the need of finding examples of explicitly calculable Martin boundaries for Markov processes. Spitzer's note only contained a brief sketch of a proof and covered only the pure death case, but he claimed that the method would work as well for arbitrary offspring distributions if $\E[Z_1\log Z_1] < \infty$. A full proof of this fact appeared in Athreya and Ney's well-known monograph \cite[p.\,69]{Athreya1972}, which also covers the Yaglom limit but does not treat QSD in general.

In the 1970's, Hoppe considered again the question of the uniqueness of the QSD with generating functions $1-(1-H(z))^\alpha$, $\alpha\in(0,1]$. As many of the previous works on branching processes, he extensively used generating functions. Starting point was the following equation, which, for the generating function $G$ of a probability measure $\nu$, is easily seen to be equivalent to \eqref{eq:gf_measure}:
\begin{equation}
\label{eq:gf}
1- G(F(z)) = \lambda(1-G(z)).
\end{equation}
Hence, finding all QSD of eigenvalue $\lambda$ amounts to finding all probability generating functions $G$ solving \eqref{eq:gf}. Hoppe \cite{Hoppe1976} showed in 1976 that one can reduce the problem\footnote{He also showed in another article  \cite{Hoppe1977} that this is true for invariant measures as well, which allowed him to prove Formula \eqref{eq:Q_rep2} without additional conditions on the offspring distribution. Note that Formula \eqref{eq:Q_rep2} was again reproven in the general case in \cite{Alsmeyer2006}, the authors of which were apparently unaware of Hoppe's work.} to the pure death case $F(z) = 1-m(1-z)$: He proves that a generating function $G$ satisfies \eqref{eq:gf} with $\lambda=m^\alpha$ if and only if there exists a generating function $A(z)$, such that $G(z) = A(H(z))$ and
\begin{equation}
\label{eq:A_eq}
1- A(1-m(1-z)) = m^\alpha(1-A(z)).
\end{equation}
He also remarks that the general solution $A(z)$ to this equation is of the form 
\begin{equation}
\label{eq:A}
A(z) = 1-(1-z)^\alpha\exp(\psi(-\log(1-z))),
\end{equation}
for a $|\log m|$-periodic function $\psi$ with $\psi(0)=0$. The drawback of this representation, apart from its uncertain probabilistic meaning, is that it is not immediate from \eqref{eq:A} whether the Taylor series of $A(z)$ only has non-negative coefficients, i.e.\ whether $A(z)$ is the generating function of a probability distribution. Hoppe \cite{Hoppe1976} was not even sure whether such a function exists for a non-constant $\psi$. However, one can show (using for example theorems by Flajolet and Odlyzko \cite[Proposition~1]{FO1990}) that for every $c_1,\ldots,c_n$ there exists $c_0 > 0$, such that for $|c| < c_0$, the Taylor expansion at 0 of the function
\[
A(z) = 1-(1-z)^\alpha\exp\left(c\sum_{k=1}^nc_k\sin\left(\frac{2\pi k}{\log m} \log(1-z)\right)\right),
\]
only has non-negative coefficients (a similar reasoning has been used by Kingman in his article cited above \cite{Kingman1965}). The function is therefore a generating function of a probability distribution which is a QSD of the pure death process. This gives an alternative proof of non-uniqueness of the QSD but no satisfying characterization.

In 1980, Hoppe \cite{Hoppe1980} therefore published another representation of solutions of \eqref{eq:A_eq}: He showed that there exists a one-to-one correspondence between QSD and invariant measures of the BGW process. Again, he used functional equations: by \eqref{eq:gf_measure}, a (non-trivial) measure $\nu$ on $\N$ is an invariant measure of the BGW process if and only if  there exists a normalizing constant $c>0$, such that the generating function $Q(z) = \sum_{n=1}^\infty c\nu(n)z^n$ satisfies the functional equation 
\begin{equation}
\label{eq:Q}
Q(F(z)) = 1+Q(z),\quad Q(0)=0.
\end{equation}
Hoppe \cite{Hoppe1980} then showed that for every $\alpha\in(0,1]$, the function\footnote{When checking this formula in \cite{Hoppe1980}, one should be careful about the typographical ambiguity there: the appearances of ``$\log mP(t)$'' should be replaced by ``$(\log m)P(t)$''.}
\begin{equation}
\label{eq:qsd_Q}
G_\alpha(z) = \frac{\int_0^z H'(w) m^{(\alpha-1) Q(w)}\,dw}{\int_0^1 H'(w)m^{(\alpha-1) Q(w)}\,dw}
\end{equation}
is the generating function of a QSD of eigenvalue $m^\alpha$ of the BGW process and conversely, for every such function, setting
\begin{equation}
\label{eq:Q_qsd}
Q(z) = \frac{\log(1-G_\alpha(z))}{\log m^{\alpha}}
\end{equation}
defines a generating function which solves \eqref{eq:Q} (note that this is a special case of the compositions of generating functions studied at the end of Section~\ref{sec:subordinators}). This yields for every $\alpha\in(0,1)$ a bijection between all QSD of eigenvalue $m^{\alpha}$ and all invariant measures and thus apparently solves the problem of characterizing all QSD. However, the non-linear transformations from Equations \eqref{eq:qsd_Q} and \eqref{eq:Q_qsd} do not seem to be easy to tame, for example, we are not aware of any direct way of obtaining a formula like \eqref{eq:G_rep} from \eqref{eq:Q_rep2} using the above formulae. More specifically, we are unable to relate the measures $\Lambda$ in the respective representations of $G_\alpha$ and $Q$ in \eqref{eq:G_rep}, when $G_\alpha$ and $Q$ are related through \eqref{eq:qsd_Q} or \eqref{eq:Q_qsd}. We do not believe that there exists a simple relation between them, similarly to our reservations concerning Question~\ref{q:lambda}. Therefore, to the best of our knowledge, the current article provides a new approach to $\lambda$-invariant measures (and, in particular, quasi-stationary distributions) of subcritical BGW processes, yielding for the first time a complete characterization of these measures involving an explicit formula.

\bibliography{qsd-gw}

\begin{thebibliography}{10}

\bibitem{Alsmeyer2006}
Gerold Alsmeyer and Uwe R{\"{o}}sler.
\newblock {The Martin entrance boundary of the Galton–Watson process}.
\newblock {\em Annales de l'Institut Henri Poincare (B) Probability and
  Statistics}, 42(5):591--606, sep 2006.

\bibitem{Athreya1972}
Krishna~B. Athreya and Peter~E. Ney.
\newblock {\em {Branching processes}}, volume 196 of {\em Die Grundlehren der
  mathematischen Wissenschaften}.
\newblock Springer-Verlag, New York, 1972.

\bibitem{Bertoin1996}
Jean Bertoin.
\newblock {\em {L{\'{e}}vy processes}}, volume 121 of {\em Cambridge Tracts in
  Mathematics}.
\newblock Cambridge University Press, Cambridge, 1996.

\bibitem{Cavender1978}
James~A. Cavender.
\newblock {Quasi-stationary distributions of birth-and-death processes}.
\newblock {\em Advances in Applied Probability}, 10(3):570--586, 1978.

\bibitem{Embrechts2002}
Paul Embrechts and Makoto Maejima.
\newblock {\em {Selfsimilar processes}}.
\newblock Princeton Series in Applied Mathematics. Princeton University Press,
  Princeton, NJ, 2002.

\bibitem{Ferrari1995a}
P.~A. Ferrari, H.~Kesten, S.~Martinez, and P.~Picco.
\newblock {Existence of Quasi-Stationary Distributions. A Renewal Dynamical
  Approach}.
\newblock {\em The Annals of Probability}, 23(2):501--521, apr 1995.

\bibitem{FO1990}
Philippe Flajolet and Andrew Odlyzko.
\newblock {Singularity analysis of generating functions}.
\newblock {\em SIAM Journal on Discrete Mathematics}, 3(2):216--240, 1990.

\bibitem{Heathcote1967}
C.R. Heathcote, E.~Seneta, and D.~Vere-Jones.
\newblock {A refinement of two theorems in the theory of branching processes}.
\newblock {\em Theory of Probability \& Its Applications}, 12(2):297--301,
  1967.

\bibitem{Henard2014}
Olivier H{\'{e}}nard and Pascal Maillard.
\newblock {On trees invariant under edge contraction}.
\newblock {\em arXiv:1403.5491}, 2014.

\bibitem{Hoppe1976}
Fred~M. Hoppe.
\newblock {On a Result of Rubin and Vere-Jones concerning Subcritical Branching
  Processes}.
\newblock {\em Journal of Applied Probability}, 13(4):804, dec 1976.

\bibitem{Hoppe1977}
Fred~M. Hoppe.
\newblock {Representations of Invariant Measures on Multitype Galton-Watson
  Processes}.
\newblock {\em The Annals of Probability}, 5(2):291--297, apr 1977.

\bibitem{Hoppe1980}
Fred~M. Hoppe.
\newblock {On a Schr{\"{o}}der equation arising in branching processes}.
\newblock {\em aequationes mathematicae}, 20(1):33--37, dec 1980.

\bibitem{Joffe1967}
Anatole Joffe.
\newblock {On the Galton-Watson branching process with mean less than one}.
\newblock {\em The Annals of Mathematical Statistics}, 38(1):264--266, 1967.

\bibitem{Kallenberg1997}
Olav Kallenberg.
\newblock {\em {Foundations of modern probability}}.
\newblock Probability and its Applications (New York). Springer-Verlag, New
  York, 1997.

\bibitem{Kemeny1976}
John~G Kemeny, J~Laurie Snell, and Anthony~W Knapp.
\newblock {\em {Denumerable Markov chains}}.
\newblock Springer-Verlag, New York-Heidelberg-Berlin, second edition, 1976.

\bibitem{Kingman1965}
J.F.C. Kingman.
\newblock {Stationary measures for branching processes}.
\newblock {\em Proceedings of the American Mathematical Society},
  16(2):245--247, 1965.

\bibitem{Lambert2007}
Amaury Lambert.
\newblock {Quasi-Stationary Distributions and the Continuous-State Branching
  Process Conditioned to be Never Extinct}.
\newblock {\em Electronic Journal of Probability}, 12(14):420--446, apr 2007.

\bibitem{Meleard2012}
Sylvie M{\'{e}}l{\'{e}}ard and Denis Villemonais.
\newblock {Quasi-stationary distributions and population processes}.
\newblock {\em Probability Surveys}, 9:340--410, 2012.

\bibitem{Pakes1999}
Anthony~G. Pakes.
\newblock {Revisiting Conditional Limit Theorems for the Mortal Simple
  Branching Process}.
\newblock {\em Bernoulli}, 5(6):969, dec 1999.

\bibitem{Rubin1968}
H.~Rubin and D.~Vere-Jones.
\newblock {Domains of Attraction for the Subcritical Galton-Watson Branching
  Process}.
\newblock {\em Journal of Applied Probability}, 5(1):216, apr 1968.

\bibitem{Seneta1971}
E.~Seneta.
\newblock {On Invariant Measures for Simple Branching Processes}.
\newblock {\em Journal of Applied Probability}, 8(1):43, mar 1971.

\bibitem{Seneta1966}
E.~Seneta and D.~Vere-Jones.
\newblock {On Quasi-Stationary Distributions in Discrete-Time Markov Chains
  with a Denumerable Infinity of States}.
\newblock {\em Journal of Applied Probability}, 3(2):403--434, dec 1966.

\bibitem{Spitzer1967}
Frank Spitzer.
\newblock {Two explicit Martin boundary constructions}.
\newblock In {\em Symposium on Probability Methods in Analysis}, number~x in
  Lecture Notes in Mathematics Vol. 31, pages 296--298. Springer Berlin, 1967.

\bibitem{Woess2000}
Wolfgang Woess.
\newblock {\em {Random walks on infinite graphs and groups}}, volume 138 of
  {\em Cambridge Tracts in Mathematics}.
\newblock Cambridge University Press, Cambridge, 2000.

\bibitem{Yaglom1947}
A~M Yaglom.
\newblock {Certain limit theorems of the theory of branching random processes}.
\newblock {\em Doklady Akad. Nauk SSSR (N.S.)}, 56:795--798, 1947.

\end{thebibliography}

\end{document}